\documentclass{amsart}
\usepackage{a4wide}

\usepackage{textcmds} 
\usepackage{amsmath, amssymb, amsfonts, amstext, amsthm, amscd, mathrsfs, mathscinet}
\usepackage{mathtools}
\usepackage{verbatim}
\usepackage{graphicx}
\usepackage{color}
\usepackage{pinlabel}
\usepackage{mathrsfs} 
\usepackage{xypic}
\usepackage{enumitem}
\usepackage{tikz}
\usetikzlibrary{matrix}

\usepackage{etoolbox}
\makeatletter
\let\ams@starttoc\@starttoc
\makeatother
\usepackage[parfill]{parskip}
\makeatletter
\let\@starttoc\ams@starttoc
\patchcmd{\@starttoc}{\makeatletter}{\makeatletter\parskip\z@}{}{}
\makeatother

\usepackage{hyperref}

\newcommand{\CC}{\mathbf{C}}

\newcommand{\RR}{\mathbf{R}}

\newcommand{\ZZ}{\mathbf{Z}}
\newcommand{\TT}{\mathbf{T}}
\newcommand{\kk}{\mathbf{k}}

\newcommand{\OP}{\operatorname}

\newsavebox{\textvisiblespacebox}
\savebox{\textvisiblespacebox}{\texttt{aa}}
\newcommand\vartextvisiblespace[1][\wd\textvisiblespacebox]{%
  \makebox[#1]{\kern.1em\rule{.4pt}{.3ex}%
  \hrulefill%
  \rule{.4pt}{.3ex}\kern.1em}%
}

\numberwithin{equation}{section}

\newtheorem{thm}{Theorem}[section]
\newtheorem{lma}[thm]{Lemma}

\newtheorem{cor}[thm]{Corollary}
\newtheorem{mainthm}{Theorem}
\newtheorem{maincor}[mainthm]{Corollary}

\newtheoremstyle{TheoremNum}
    {\topsep}{\topsep}              
    {\itshape}                      
    {}                              
    {\bfseries}                     
    {.}                             
    { }                             
    {\thmname{#1}\thmnote{ \bfseries #3}}
\theoremstyle{TheoremNum}

\theoremstyle{definition}
\newtheorem{dfn}[thm]{Definition}

\theoremstyle{remark}
\newtheorem{rmk}[thm]{Remark}

\begingroup 
\makeatletter 
\@for\theoremstyle:=definition,remark,plain,TheoremNum\do{%
\expandafter\g@addto@macro\csname th@\theoremstyle\endcsname{%
\addtolength\thm@preskip\parskip 
}%
} 
\endgroup 

\title[Legendrians with unbounded spectral norm]{Families of Legendrians and Lagrangians with unbounded spectral norm}

\author{Georgios Dimitroglou Rizell}
\address{Department of Mathematics\\
Uppsala University\\
Box 480\\
SE-751 06 UPPSALA\\
SWEDEN}
\email{georgios.dimitroglou@math.uu.se}
\thanks{The author is supported by the grant KAW 2016.0198 from the Knut and Alice Wallenberg Foundation.}

\begin{document}

\begin{abstract}
Viterbo has conjectured that any Lagrangian in the unit co-disc bundle of a torus which is Hamiltonian isotopic to the zero-section satisfies a uniform bound on its spectral norm; a recent result by Shelukhin showed that this is indeed the case. The modest goal of this note is to explore two natural generalisations of the above geometric setting in which the bound of the spectral norm fails: first, passing to Legendrian isotopies in the contactisation of the unit co-disc bundle (Hamiltonian isotopies lift to Legendrian isotopies) and, second, considering Hamiltonian isotopies but after modifying the co-disc bundle by attaching a critical Weinstein one-handle.
\end{abstract}

\maketitle
\setcounter{tocdepth}{1}
\tableofcontents

\section{Introduction and results}

Spectral invariants were introduced in Viterbo's seminal work \cite{Viterbo:Generating}. Since their appearance they have become one of the most fundamental tools of quantitative symplectic topology. We do not intend to give an overview of its development and many applications here; instead we direct the reader to work by Oh \cite{Oh:Spectral} for a thorough introduction to the subject from a modern perspective.

Very briefly, spectral invariants in the symplectic case consists of functions from Hamiltonian diffeomorphisms
$$c \colon Ham(X,\omega) \to \RR$$
that take values in the real numbers, and which satisfy a list of axioms that will be omitted. The spectral invariants that we consider here are constructed as follows. For a pair of exact Lagrangian submanifolds $L_0,L_1 \subset (X,d\lambda)$ (the symplectic manifold is thus necessarily exact) one can associate the Floer complex $CF(L_0,\phi(L_1))$ to any Hamiltonian diffeomorphism $\phi \in Ham(X,\omega)$ endowed with its canonical action filtration. Spectral invariants are certain real numbers which encode information about this filtered chain complex. In order to make this precise, we utilise the language of the {\bf barcode} from persistent homology from topological data analysis, which goes back to work by Carlsson--Zomorodian-Collins--Guibas \cite{Carlsson}; see Section \ref{sec:barcode} for our description and also \cite{Polterovich:Persistence} for a thorough introduction. The barcode can be defined for any chain complex $(C,\partial,\mathfrak{a})$ with a filtration
by subcomplexes $\mathfrak{a}^{-1}[-\infty,c] \subset C_*$ defined by an "action" function
$$\mathfrak{a} \colon C \to \{-\infty\} \cup \RR.$$
Phrased in this language, the spectral invariants are the values of the starting points of the semi-infinite bars of the barcode associated to the Floer complex. In fact, the main interest here is not the spectral invariants themselves, but rather the following derived concepts (see Definition \ref{dfn:main}):
\begin{itemize}
\item The {\bf spectral range} of a filtered complex, denoted by
$$\rho(C,\partial,\mathfrak{a}) \in \{-\infty\} \cup [0,+\infty].$$
This quantity is defined as the maximal distance between the starting points of two semi-infinite bars in the corresponding barcode.
\item The {\bf boundary depth} of a filtered complex, denoted by
$$\beta(C,\partial,\mathfrak{a}) \in \{-\infty\} \cup [0,+\infty].$$
This quantity is defined as the maximal length of a finite bar in the corresponding barcode.
\end{itemize}
For the Floer complex $CF(L,\phi^1_H(L))$ of a Lagrangian and its Hamiltonian deformation, the spectral range coincides with the {\bf spectral norm}, which we define as
$$ \gamma( CF(L,\phi^1_H(L)) \coloneqq \rho(CF(L,\phi^1_H(L)).$$
In general, the spectral norm can be defined whenever the complex satisfies Poincar\'{e} duality in a certain technical sense. (Since we do go into the details of the axioms of spectral invariants here, the difference between the concept of spectral range and spectral norm necessarily becomes obscure.)

We also need a generalisation of the above spectral invariants to contact manifolds. Since we will only consider with contact manifolds of a very particular type, namely contactisations
$$(Y,\alpha)=(X \times \RR,dz+\lambda)$$
of exact symplectic manifolds $(X,d\lambda)$ (see Section \ref{sec:contact}), this can be done by relying on well-established techniques. From our point of view, the spectral invariants of a contact manifold are defined for the group of contactomorphisms which are contact-isotopic to the identity, and yield functions
$$ c \colon Cont_0(Y,\alpha) \to \RR.$$
Note that the value does depend on the choice of contact form $\alpha$ here, and not just on the contact structure $\ker\alpha \subset TY$. It should be noted that spectral invariants in the contact setting are much less studied and developed than the symplectic version. However, the original formulation of the spectral invariants, which appeared in \cite{Viterbo:Generating} for symplectic cotangent bundles $(X,\omega)=(T^*M,d(p\,dq))$, admits a straight-forward generalisation to the standard contact jet-space
$$(J^1M=T^*M\times \RR,dz-p\,dq),$$
as shown by Zapolsky \cite{Zapolsky:Jet}. In fact, the spectral invariants in \cite{Viterbo:Generating} are based on a version of Floer homology defined using generating families, and that this theory can be generalised to invariants of Legendrian isotopies inside jet-spaces by work of Chekanov \cite{Chekanov:Generating}. Note that jet-spaces are particular cases of contactisations.

The spectral invariants considered here can be defined either by using generating families as in \cite{Zapolsky:Jet}, or by using a Floer homology constructed by using the Chekanov--Eliashberg algebra as first done in \cite{DualityLeg} by Ekholm--Etnyre--Sabloff; also see work \cite{Dimitroglou:Cthulhu} by the author together with Chantraine--Ghiggini--Golovko. (Strictly speaking, all axioms of the spectral invariants have not yet been established in the latter setting, but this does not affect the results here.) More precisely, given a pair of Legendrians $\Lambda_0$ and $\Lambda_1$, the spectral invariants are associated to the barcode of the Floer complex $CF(\Lambda_0,\phi(\Lambda_1))$ where $\phi$ is a contactomorphism which is contact isotopic to the identity. In fact, the Floer homology for the exact Lagrangians that we use here will be defined by using exactly the same technique; we lift the exact Lagrangian submanifold to a Legendrian submanifold of the corresponding contactisation, and then use the Floer complex in the contact setting. More details are given in Section \ref{sec:floer}.

Viterbo conjectured in \cite{Viterbo:Homogen} that the spectral norm $\gamma(CF(0_{\TT^n},\phi(0_{\TT^n})))$ of the Floer complex of the zero section $0_\TT^n \subset T^*\TT^n$ satisfies a uniform bound whenever $\phi \in Ham(T^*\TT^n)$ maps the zero section $\phi(0_{\TT^n}) \subset DT^*\TT^n$ into the unit cotangent bundle. In recent work by Shelukhin \cite{Shelukhin:Viterbo1}, \cite{Shelukhin:Viterbo2} this property was finally shown to be the case, even for a wide range of cotangent bundles beyond the torus case. The main point of our work here is to give examples of geometric settings beyond symplectic co-disc bundles, where the analogous boundedness of the spectral norm fails. It should be stressed that, in the time of writing of this article, there are still many cases of cotangent bundles for which the original formulation of the problem remains open: does the spectral norm of an exact Lagrangian inside $DT^*M$ which is Hamiltonian isotopic to the zero-section satisfy a uniform bound for any closed smooth manifold $M$?

As a first result, in Part (1) of Theorem \ref{mainthm}, we show that the spectral norm of Legendrians inside the contactisation $D^*S^1 \times \RR \subset J^1S^1$ which are Legendrian isotopic to the zero section does not satisfy a uniform bound. Recall that any Hamiltonian isotopy of $0_{S^1} \subset DT^*S^1$ lifts to a Legendrian isotopy of the zero section $j^10 \subset J^1S^1$ (see Lemma \ref{lma:lift}); consequently, one way to formulate Part (1) of Theorem \ref{mainthm} is by saying that Viterbo's conjecture cannot be generalised to Legendrian isotopies.

Below we denote by
$$F_{\theta_0,z_0}=\{\theta=\theta_0,z=z_0\} \subset J^1S^1$$
the Legendrian lift of a cotangent fibre $T^*_{\theta_0}S^1$.
\begin{mainthm}
\label{mainthm}
\begin{enumerate}
\item There exists a Legendrian isotopy of the zero section $j^10 \subset J^1S^1$ which satisfies
$$\phi^t \colon j^10 \hookrightarrow DT^*S^1 \times \RR=(S^1 \times [-1,1]) \times \RR,$$
and for which $CF(j^10,\phi^t(j^1))$ all are generated by precisely two mixed Reeb chords, whose difference in action moreover becomes arbitrarily large as $t\to+\infty$. In particular, the spectral norm $\gamma(CF(j^10,\phi^1(j^10)))$ becomes arbitrarily large as $t \to +\infty$.
\item There exists a Legendrian isotopy of the standard unknot $\Lambda_0 \subset J^1S^1$ shown in Figure \ref{fig:c} which satisfies
$$\phi^t \colon \Lambda_0 \hookrightarrow DT^*S^1 \times \RR_{>0}=(S^1 \times [-1,1]) \times \RR_{>0} \subset J^1S^1,$$
and where the boundary depth $\beta(CF(\phi^1(\Lambda_0),F_{\theta_0,z}))$ becomes arbitrarily large as $t\to+\infty$. In addition, we may assume that $\phi^t$ is supported inside some subset $\{z \ge c\}$ for which the inclusion $\{z \ge c \} \cap \Lambda_0 \subsetneq \Lambda_0$ is a strict subset. 
\end{enumerate}
\end{mainthm}
In recent work \cite[Section 6.2]{Biran:Bounds} Biran--Cornea showed that a bound $\gamma(CF(0^M,L)) \le C$ on the spectral norm of the Floer complex of a Lagrangian $L \subset T^*M$, where $L$ is Hamiltonian isotopic to the zero section, implies the bound $\beta(CF(L,T^*_{pt}M)) \le 2C$ on the boundary depth of the Floer complex of $L$ and a cotangent fibre. The Legendrians produced by Part (2) of Theorem \ref{mainthm} can be used to show that the analogous result cannot be generalised to Legendrian isotopies. More precisely,
\begin{maincor}
There exists a Legendrian isotopy of the zero section $j^10 \subset J^1S^1$ that satisfies
$$\phi^t \colon j^10 \hookrightarrow DT^*S^1 \times \RR=(S^1 \times [-1,1]) \times \RR,$$
and for which the spectral norm $\gamma(CF(j^10,\phi^1(j^10))$ is uniformly bounded for all $t \ge 0$, while the boundary depth $\beta(CF(\phi^1(0_{S^1}),F_{\theta_0,z_0}))$ becomes arbitrarily large as $t \to +\infty$.
\end{maincor}
\begin{proof}
Take a cusp-connected sum of a $C^1$-small perturbation of the zero-section $j^10$ and any unknot $\Lambda_t$ from the family produced by Part (2) of Theorem \ref{mainthm}; the case of $\Lambda_0$ is shown in Figure \ref{fig:c}. We refer to \cite{Dimitroglou:Ambient} for the definition of cusp-connected sum (also called ambient Legendrian 0-surgery) along a Legendrian arc (the so-called surgery disc). We perform the cusp-connected sum along a Legendrian arc which is contained inside the region $\{z<c\}$ which is disjoint from the support of the Legendrian isotopy of the unknots. Note that the Legendrian resulting from the cusp-connected sum is Legendrian isotopic to the zero-section, as shown in Figure \ref{fig:c}. It follows that the same is true for the cusp-connected sum of $j^10$ and any Legendrian $\Lambda_t$ from the family.

Finally, in order to evaluate the effect of the ambient surgery on the barcodes of the Floer complexes we apply Theorem \ref{thm:ambientsurgery}. To that end, the following two facts are needed. First, $CF(\Lambda_0,j^10)$ is acyclic, and thus its barcode consists of only finite bars. This follows by the invariance under Legendrian isotopy. (After a translation of $\Lambda_0$ sufficiently far in the negative $z$-direction, all generators of the Floer complex disappear.)  Second,
$$CF(j^1f \cup \Lambda_t, j^10)=CF(j^1f,j^10)\oplus CF(\Lambda_t,j^10)$$
is a direct sum of complexes. The barcode is thus the union of barcodes.
\end{proof}
\begin{figure}[htp]
	\vspace{3mm}
	\labellist
	\pinlabel $\color{blue}\Lambda_0$ at 75 50
	\pinlabel $\color{blue}j^10$ at 75 19
	\pinlabel $\color{blue}\Lambda_-$ at 225 32
	\pinlabel $\color{blue}\Lambda_+$ at 360 32
	\pinlabel $z$ at 56 89
	\pinlabel $z$ at 192 89
	\pinlabel $z$ at 327 89
	\pinlabel $\theta$ at 122 27
	\pinlabel $\theta$ at 257 27
	\pinlabel $\theta$ at 392 27
	\pinlabel $c$ at 63 36
	\pinlabel $\frac{1}{2}$ at 110 35
	\pinlabel $-\frac{1}{2}$ at 0 35
	\endlabellist
	\includegraphics{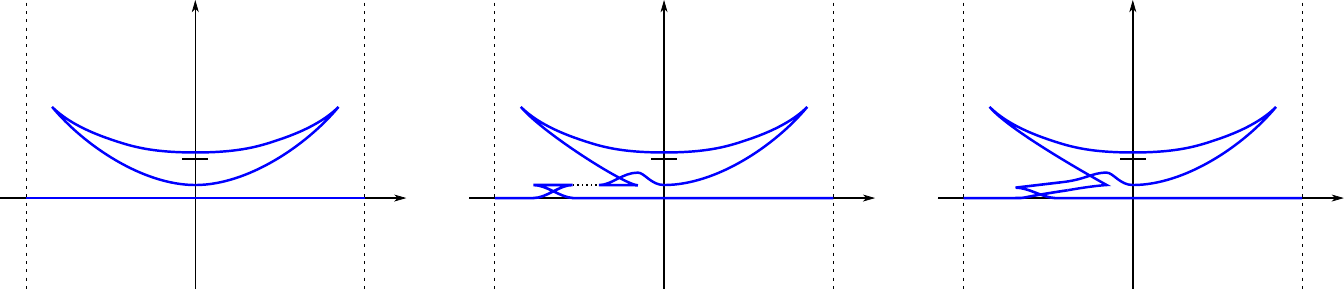}
	\caption{\emph{Left:} the front projection of the zero section $j^10 \subset J^1(\RR/\ZZ)=J^1S^1$ and a standard Legendrian unknot $\Lambda_0$. \emph{Middle:} The result of a Legendrian \emph{RI}-move on each component, $\Lambda_-$ denotes the union of the two components. \emph{Right:} The Legendrian $\Lambda_+$ which is the result after a cusp-connected sum along the dotted arc shown in the middle picture. $\Lambda_+$ is Legendrian isotopic to the zero section ($\Lambda_+$ is obtained by performing two \emph{RI}-moves on the zero-section).}
	\label{fig:c}
\end{figure}
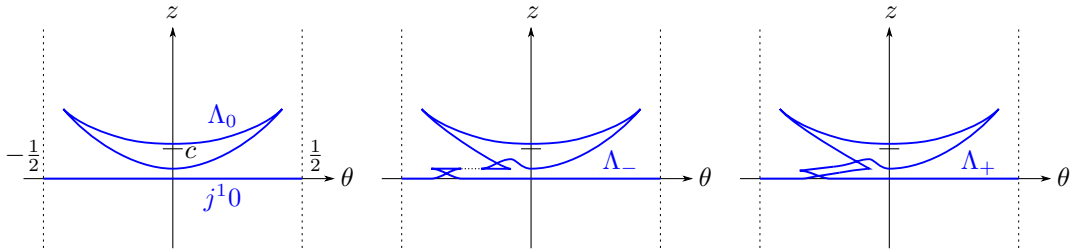

\begin{mainthm}
\label{thm:ambientsurgery}
If $\Lambda_+$ is a Legendrian obtained from $\Lambda_-$ by a Legendrian ambient surgery. After making the surgery-region sufficiently small, we can assume that there is an action-preserving isomorphism
$$CF((\Lambda_+,\varepsilon_+),(\Lambda,\varepsilon)) \to CF((\Lambda_-,\varepsilon_-),(\Lambda,\varepsilon))$$
of complexes, where $(\Lambda,\varepsilon)$ is an arbitrary but fixed Legendrian, and where the augmentation $\varepsilon_+$ is induced by pulling back the augmentation $\varepsilon_-$ under the DGA-morphism induced by the standard Lagrangian handle-attachment cobordism. In particular, the barcodes of the two Floer complexes coincide.
\end{mainthm}
In the setting of exact Lagrangian cobordisms in the sense of Arnol'd between exact Lagrangian submanifolds similar results were found in \cite[Section 5.3]{Biran:Bounds} .

Finally we present a Hamiltonian isotopy of a closed exact Lagrangian inside a Liouville domain for which the spectral norm becomes arbitrarily large. The simplest examples of such a Liouville domain is the 2-torus with an open ball removed; we denote this by $(\Sigma_{1,1},d\lambda)$ and depict it in Figure \ref{fig:t1}. The detailed construction is given in Section \ref{sec:torus}.

\begin{mainthm}
\label{mainthm:torus}
There exists a closed exact Lagrangian submanifold $L \subset (\Sigma_{1,1},\omega)$ and a compactly supported Hamiltonian $H\colon \Sigma_{1,1} \to \RR$ for which the induced compactly supported Hamiltonian isotopy $\phi^t_{H} \colon (\Sigma_{1,1},\omega) \to (\Sigma_{1,1},\omega)$
satisfies the property that the spectral norm $\gamma(CF(L,\phi^t_{H}(L)))$ becomes arbitrarily large as $t \to +\infty$.
\end{mainthm}

\subsection{Why the proofs of uniform bounds fail for Legendrians}
\label{sec:why}

The techniques that are used in \cite{Shelukhin:Viterbo2} and \cite{Biran:Bounds} to prove the results in the case of the cotangent bundle are not yet fully developed in the case of Legendrians in contactisations. This includes the closed-open map, which is a crucial ingredient in \cite{Shelukhin:Viterbo2}, and a unital $A_\infty$-structure on the Floer complex with relevant PSS-isomorphisms, which is crucial in \cite{Biran:Bounds}. Nevertheless, we still do expect that these operations can be defined also for the Floer homology of Legendrians in contactisations. In fact the $A_\infty$-structure was recently extended to this setting by Legout \cite{Legout}. So this should not be the reason why the proofs break down. So, what then goes wrong in the proofs if one tries to generalise to the Legendrian case?

First we recall the properties of the Floer homology complex of a Legendrian and itself; see e.g.~\cite{DualityLeg} for the details. In order to define $CF(\Lambda,\Lambda)$ one must first make the mixed Reeb chords transverse by a Legendrian perturbation of the second copy of $\Lambda$. We do this by replacing $\Lambda$ with a section $j^1f$ in its standard contact jet-space neighbourhood, where $f \colon \Lambda \to \RR$ is a $C^1$-small Morse function. In this manner we obtain
$$CF(\Lambda,\Lambda)=C^{\mathrm{Morse}}(f;\kk) \oplus \bigoplus_{a \in \mathcal{Q}(\Lambda)} \kk p_c \oplus \kk q_c $$
where $\mathcal{Q}(\Lambda)$ denotes the set of Reeb chords on $\Lambda$, and $C^{\mathrm{Morse}}(f;\kk)$ is the Morse homology complex with basis given by the critical points of the function $f \colon \Lambda \to \RR$. The action of the former chords are approximately equal to $\mathfrak{a}(p_c)=\ell(c)$ and $\mathfrak{a}(q_c)=-\ell(c)$ while the action of the latter is equal to $\mathfrak{a}(x)=f(x)$. What is important to notice here is that the latter generators may be assumed to have arbitrarily small action, while this is not the case for the generators that correspond to pure Reeb chords. When $\Lambda$ is the Legendrian lift of a Lagrangian embedding, there are of course only chords of the former type. This turns out to be the crucial difference between the symplectic and the contact case.

\emph{Example in Part (1) of Theorem \ref{mainthm}:} The proof in \cite{Shelukhin:Viterbo2} uses the closed-open map. More precisely, a crucial ingredient in the proof is the action-preserving property of the operations $P'_{a}$ on the Floer homology $CF(0_M,\phi^1_H(0_M))$, which are defined using the length-0 part $\phi^0(a)$ and length-1 part $\phi^1(a,\cdot)$ of the closed open map for certain elements $a \in SH(T^*M)$ in symplectic homology. In the case when the Legendrian has pure Reeb chords (i.e.~it is not the lift of an exact Lagrangian embedding), the chain $\phi^0(a) \in CF(\Lambda,\Lambda)$ may thus consist of generators whose action does not vanish (since they do not correspond to Morse generators). In this case the action-preserving property of $P'_{a}$ in terms of merely the action of the element $a \in SH(T^*M)$ is lost.

\emph{Example in Part (2) of Theorem \ref{mainthm}:} The proof in \cite[Section 6.2]{Biran:Bounds} uses the fact that there are continuation elements $a \in CF(\phi^1_H(0_M),0_M)$ and $b \in CF(0_M,\phi^1_H(0_M))$ for which $\mu_2(a,b) \in CF(\phi^1_H(0_M),\phi^1_H(0_M))$ is the unique maximum of a suitable Morse function. In the Legendrian case the element $\mu_2(a,b) \in CF(\phi^1(j^10),\phi^1(j^10))$ is still a homology unit; however, it not necessarily a Morse chord, but can be of significant action. In particular, multiplication with the element $\mu_2(a,b)$ is not necessarily identity on the chain level, nor is it necessarily homotopic to the identity by a chain homotopy of small action. The geometrically induced chain homotopy $\mu_3(a,b,\cdot)$ between $\mu_2(a,\mu_2(b,\cdot))$ and $\mu_2(\mu_2(a,b),\cdot)$ increases action by at most the spectral norm, and is used in\cite{Biran:Bounds} for establishing the bound on the boundary depth. However, this chain homotopy does not to the job anymore, since we also need an additional chain-homotopy (of unknown action properties) to take us from the map $\mu_2(\mu_2(a,b),\cdot)$ to the chain level identity.

\section{Background}

\subsection{Contact geometry of jet-spaces and contactisations}
\label{sec:contact}

An {\bf exact symplectic manifold} is a smooth $2n$-dimensional manifold $(X^{2n},d\lambda)$ equipped with a choice of a primitive one-form $\lambda$ for an exact symplectic two-form $\omega=d\lambda$, i.e.~$\omega$ is skew-symmetric, non-degenerate, and closed. Note that the primitive $\lambda$ should be considered as part of the data describing the exact symplectic manifold. A compact exact symplectic manifold with boundary $(W,d\lambda)$ is a {\bf Liouville domain} if the {\bf Liouville vector field}, i.e.~the vector field $\zeta$ given as the symplectic dual of $\lambda$ via the equation $\iota_\zeta d\lambda=\lambda$, is transverse to the boundary $\partial W$. The flow generated by $\zeta$ is called the {\bf Liouville flow} and satisfies $(\phi^t_\zeta)^*\lambda=e^t\lambda$. An open exact symplectic manifold $(\overline{W},d\lambda)$ is a {\bf Liouville manifold} if the all critical points of the Liouville vector field are contained inside some compact Liouville domain $W \subset (\overline{W},d\lambda)$, and if the Liouville flow is complete.

A {\bf Hamiltonian isotopy} is a smooth isotopy of $X$ which is generated by a time-dependent vector field $V_t \in \Gamma(TX)$ that satisfies $\iota_{V_t}d\lambda=-dH_t$ for some smooth time-dependent function
$$H \colon X \times \RR_t \to \RR$$
which is called the {\bf Hamiltonian}; a diffeomorphism of $X$ which is the time-$t$ flow generated by such a vector field preserves the symplectic form (but not the primitive) and is denoted by
$$\phi^t_{H} \colon (X,\omega) \to (X,\omega);$$
we call such a map a {\bf Hamiltonian diffeomorphism}, and the corresponding flow a {\bf Hamiltonian isotopy}. Conversely, any choice of Hamiltonian function induces a {\bf Hamiltonian isotopy} $\phi^t_H$ in the above manner. Since we consider exact symplectic manifolds, a smooth isotopy $\phi^t \colon X \to X$ is a Hamiltonian isotopy if and only if $(\phi^t)^*\lambda=\lambda+dG_t$ holds for some smooth function
$$G \colon X \times \RR_t \to \RR.$$
Note that the Hamiltonian function that corresponds to a Hamiltonian isotopy is determined only up to the addition of a constant.

Any exact $2n$-dimensional symplectic manifold $(X^{2n},d\lambda)$ gives rise to a $2n+1$-dimensional contact manifold $(X \times \RR_z,dz+\lambda)$ called its {\bf contactisation}, which is equipped with the canonical contact one-form $\alpha_{st}\coloneqq dz+\lambda$. The contactisations induced by choices of primitives of the symplectic form $\lambda$ and $\lambda'=\lambda+df$ that differ by the exterior differential of $f \colon X \to \RR$ are isomorphic via the coordinate change $z\mapsto z-f$. Recall that the contact condition is equivalent to $d\alpha_{st}$ being non-degenerate on the contact planes $\ker \alpha_{st} \subset T(X \times \RR)$. A {\bf contact isotopy} is a smooth isotopy which preserves the distribution $\ker \alpha_{st}$ (but not necessarily the contact form). The contraction $\iota_{V_t}\alpha_{st}$ of the contact form and the infinitesimal generator gives a bijective correspondence between contact isotopies and smooth time-dependent functions on $X \times \RR$, the latter are called {\bf contact Hamiltonians}. We refer to \cite{Geiges:Intro} for more details. 
\begin{lma}
\label{lma:lift}
A Hamiltonian isotopy $\phi^t_H \colon (X,d\lambda) \to (X,d\lambda)$ with a choice of Hamiltonian $H_t \colon X \to \RR$ lifts to a contact isotopy
\begin{gather*}
X \times \RR \to X \times \RR,\\
(x,z) \mapsto (\phi^t_H(x),z-G_t(x)),
\end{gather*}
where the function $G \colon X \times \RR_t \to \RR$ is defined by
$$G_t(x)=\int_0^t \lambda(V_s(\phi^s_{H}(x)))-H_s(\phi^s_H(x))ds$$
and satisfies the property
$$(\phi^t_{H})^*\lambda=\lambda+d\,G_t.$$
Moreover, this contact isotopy preserves the contact form $\alpha_{st}$ and is generated by the time-dependent contact Hamiltonian $H_t \circ \OP{pr}_X \colon X \times \RR_z \to \RR$.
\end{lma}

A smooth immersion of an $n$-dimensional manifold
$$\Lambda \looparrowright (X^{2n} \times \RR,dz+\lambda)$$
in the contactisation is {\bf Legendrian} if it is tangent to $\ker \alpha_{st}$, while a smooth $n$-dimensional immersion $L \looparrowright (X^{2n},\lambda)$ in an exact symplectic manifold is {\bf exact Lagrangian} if $\lambda$ pulls back to an \emph{exact} one-form. The following relation between Legendrians and exact Lagrangians is immediate:
\begin{lma}
\label{lma:laglift}
The canonical projection of a Legendrian immersion to $(X,\lambda)$ is an exact Lagrangian immersion. Conversely, any exact Lagrangian immersion lifts to a Legendrian immersion of the contactisation $X \times \RR$. Moreover, the lift is uniquely determined by the choice of a primitive $f \colon L \to \RR$ of the pull-back $\lambda|_{TL}=df$, via the formula $z=-f$.
\end{lma}

Transverse double points of Lagrangian immersions are stable. On the other hand, generic Legendrian immersions are in fact \emph{embedded}. However, there are stable self-intersections of Legendrians that appear in one-parameter families. Recall the following standard fact; again we refer to e.g.~\cite{Geiges:Intro} for details.
\begin{lma}
A compactly supported smooth isotopy $\phi^t(\Lambda) \subset X \times \RR$ through Legendrian embeddings can be generated by an ambient contact isotopy.
\end{lma}

\subsubsection{The cotangent bundle and jet-space}

There is a canonical exact symplectic two-form $-d(p\,dq)$ on any smooth cotangent bundle $T^*M$, whose primitive $-p\,dq$ is the tautological one-form with a minus sign. The cotangent bundle is a Liouville manifold and any co-disc bundle is a Liouville domain. The zero-section $0_M \subset T^*M$ is obviously an exact Lagrangian embedding.

The contactisation of $T^*M$ is the one-jet space $J^1M=T^*M \times \RR_z$ with the canonical contact one-form $dz-p\,dq$. The zero-section in $T^*M$ lifts to the one jet $j^1c$ of any constant function $c$ (obviously the one jet $j^1f$ of an arbitrary function $f \colon M \to \RR$ is Legendrian isotopic to $j^10$). For us the most relevant example is actually the two-dimensional symplectic cotangent bundle $T^*S^1=S^1 \times \RR_p$ equipped with the exact symplectic two-form $-d(p\,d\theta)$, and its corresponding contactisation, i.e.~the three-dimensional contact manifold
$$(J^1S^1=T^*S^1 \times \RR_z,dz-p\,d\theta)$$
(note the sign convention for the Liouville form). 

In order to describe Legendrians in $J^1M$ we will make use of the {\bf front-projection}, by which one simply means the canonical projection
$$\Pi_F \colon J^1M \to M \times \RR_z.$$
A Legendrian immersion can be uniquely determined by its post-composition with the front projection. A generic Legendrian knot in $J^1S^1$ has a front projection whose singular locus consists of
\begin{itemize}
\item non-vertical cubical cusps and
\item transverse self-intersections.
\end{itemize}
On the other hand, note that the front projection has no vertical tangencies by the Legendrian condition.

Two sheets of the front projection that have the same slopes (i.e.~$p$-coordinates) above some given point in the base, project to a double-point inside $T^*M$. There is a bijection between double points of this projection and Reeb chords, where a Reeb chord is an integral curve of $\partial_z$ with both endpoints on the Legendrian. The difference of $z$-coordinate of the endpoint and starting point of a Reeb chord $c$ is called its {\bf length} and is denoted by $\ell(c)\ge 0$.

Double-points of the Legendrian the immersion itself corresponds to self-tangencies of the front projection. This is not a stable phenomenon, and double-points of Legendrians generically occur only in one-parameter families. These double-points can be considered as Reeb chords of length zero.

Two Legendrian knots inside $J^1\RR$ or $J^1S^1$ with generic fronts are Legendrian isotopic if and only if their front projections can be related by a sequence of \emph{Legendrian Reidemeister moves} \cite{LegendrianReidemeister} together with an ambient isotopy of the front inside $S^1 \times \RR_z$; see \cite{Etnyre:Legendrian} for an introduction to Legendrian knots. 

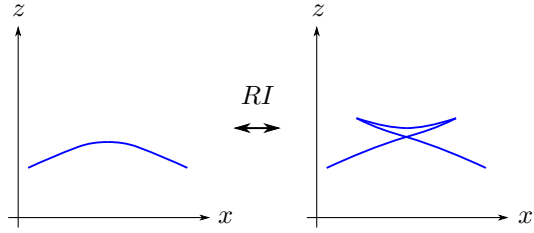
\begin{figure}[htp]
	\vspace{3mm}
	\labellist
	\pinlabel $z$ at 4 82
	\pinlabel $z$ at 117 82
	\pinlabel $RI$ at 94 50
	\pinlabel $x$ at 82 4
	\pinlabel $x$ at 195 4
	\endlabellist
	\includegraphics{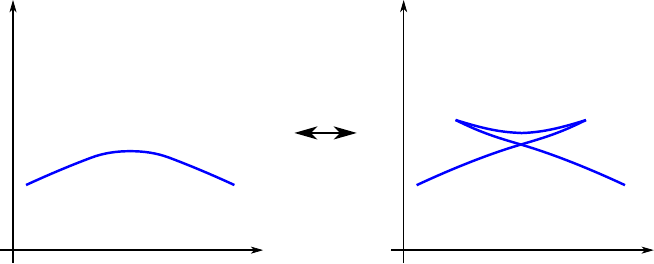}
	\caption{\emph{RI}: The first Legendrian Reidemeister move in the front projection.}
	\label{fig:r1}
\end{figure}
\begin{figure}[htp]
	\vspace{3mm}
	\labellist
	\pinlabel $z$ at 4 82
	\pinlabel $z$ at 117 82
	\pinlabel $RII$ at 94 50
	\pinlabel $x$ at 82 4
	\pinlabel $x$ at 195 4
	\endlabellist
	\includegraphics{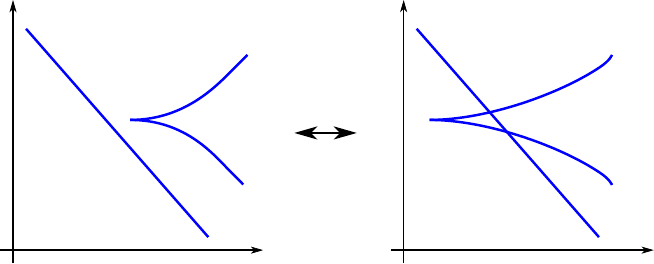}
	\caption{\emph{RII}: The second Legendrian Reidemeister move in the front projection.}
	\label{fig:r2}
\end{figure}

\begin{figure}[htp]
	\vspace{3mm}
	\labellist
	\pinlabel $z$ at 4 82
	\pinlabel $z$ at 117 82
	\pinlabel $RIII$ at 93 50
	\pinlabel $x$ at 82 4
	\pinlabel $x$ at 195 4
	\endlabellist
	\includegraphics{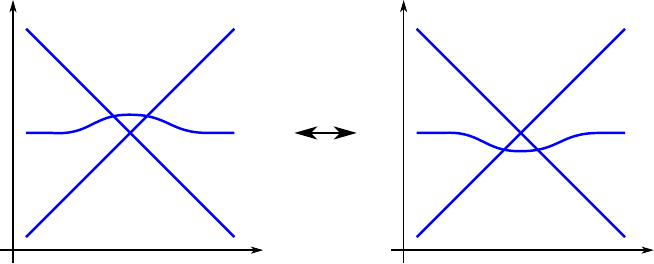}
	\caption{\emph{RIII}: The third Legendrian Reidemeister move in the front projection.}
	\label{fig:r3}
\end{figure}

For convenience we will also introduce a composite move that we will make repeated use of; this is the one shown in Figure \ref{fig:r}, which involves taking two cusp edges with different slopes, and making them cross each other (it is important that the cusps have different slopes).

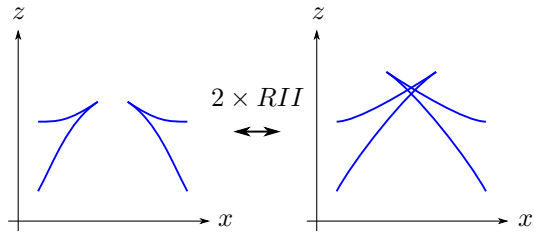
\begin{figure}[htp]
	\vspace{3mm}
	\labellist
	\pinlabel $z$ at 4 82
	\pinlabel $z$ at 117 82
	\pinlabel $2\times RII$ at 94 50
	\pinlabel $x$ at 82 4
	\pinlabel $x$ at 195 4
	\endlabellist
	\includegraphics{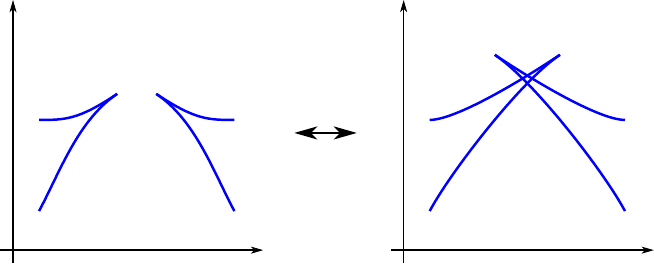}
	\caption{A composite move: the front to the right is obtained by performing two consecutive \emph{RII}-moves on the front to the left together with an isotopy.}
	\label{fig:r}
\end{figure}

\subsubsection{The punctured torus}
\label{sec:torus}
Here we construct an example of a two-dimensional non-planar Liouville domain: the two torus minus an open ball, which we denote by $(\Sigma_{1,1},d\lambda).$

First, consider the primitive
$$\lambda_0=\frac{1}{2}(p\,dq-q\,dp)$$
of the standard linear symplectic form $dp\wedge dq$ on $\RR^2.$ We have the identities
\begin{align*}
& \lambda_0-d\left(\frac{pq}{2}\right)=-q\,dp,\\
& \lambda_0+d\left(\frac{pq}{2}\right)=p\,dq.
\end{align*}
Take a smooth function $\sigma \colon \RR^2 \to \RR$ which in the standard coordinates labelled by $(p,q) \in \RR^2$ is given by
\begin{itemize}
\item $\sigma(p,q)=-pq/2$ on $\{|q| \le 1,|p|>2\}$, while it is of the form $-g(p)q/2$ for some smooth function $g$ with $g(p),g'(p) \ge 0$ on $\{|q| \le 1, |p| \ge 1\}$;
\item $\sigma(p,q)=pq/2$ on $\{|q|>2,|p| \le 1\}$, while it is of the form $g(q)p/2$ for some smooth function $g$ with $g(q),g'(q) \ge 0$ on $\{|q| \le 1, |p| \ge 1\}$;
\item $\sigma(p,q)=0$ on $\{|q|<1,|p|<1\}$; and
\end{itemize}
Consider the exact symplectic manifold $(X,d\lambda)$ which is obtained by taking the cross-shaped domain
$$ \{ p \in [-2,2], q \in [-1,1] \} \cup \{ q \in [-2,2], p \in [-1,1] \} \subset \RR^2$$
and identifying $\{p=2\}$ with $\{p=-2\}$, and $\{q=-2\}$ with $\{q=2\}$ in the obvious manner. Topologically the result is a punctured torus. The Liouville form $\lambda_0+d\sigma$ on $\RR^2$ extends to a Liouville form $\lambda$ on this punctured torus. The punctured torus has a skeleton $Sk \subset X$ which is the image of the cross $\{pq=0\}$ under the quotient; in other words, $Sk \subset X$ is the union of two smooth Lagrangian circles that intersect transversely in a single point. Note that
$$ Sk=\bigcap_{T=1}^\infty \phi^{-T}(X).$$
We claim that the sought Liouville domain $(\Sigma_{1,1},d\lambda)$ can be realised as a suitable subset of this exact symplectic manifold, simply by smoothing its corners; see Figure \ref{fig:t1}.

Since $(\Sigma_{1,1},\lambda)$ is surface with non-empty boundary, it admits a symplectic trivialisation of its tangent bundle. This implies that the all Lagrangian submanifolds of $\Sigma_{1,1}$ have a well-defined Maslov class; see Section \ref{sec:Maslov} for more details. We will make heavy use of the fact that the Maslov class depends on the choice of a symplectic trivialisation; in this case, symplectic trivialisations up to homotopy can be identified with homotopy classes of maps
$$\Sigma_{1,1} \to S^1$$
i.e.~cohomology classes $H^1(\Sigma_{1,1};\ZZ).$

\subsection{Barcode of a filtered complex and notions from spectral invariants}
\label{sec:barcode}

A {\bf filtered complexes} over some field $\kk$ is a chain complex $(C,\partial,\mathfrak{a})$ in which each element is endowed with an action $\mathfrak{a}(c) \in \RR \sqcup \{-\infty\}$ and such that the following properties are satisfied: 
\begin{itemize}
\item $\mathfrak{a}(c)=-\infty$ if and only if $c=0$,
\item $\mathfrak{a}(r\cdot c)=\mathfrak{a}(c)$ for any $r \in \kk^*$,
\item $\mathfrak{a}(a+b) \le \max\{\mathfrak{a}(a),\mathfrak{a}(b)\}$, and
\item $\mathfrak{a}(\partial(a))<\mathfrak{a}(a)$ for any $a \neq 0$.
\end{itemize}
The subset 
$$C^{<a}=\mathfrak{a}^{-1}(\{-\infty\} \cup (-\infty,a))$$
is a $\kk$-subspace by the first three bullet points; this subspace is a subcomplex by the last bullet point.  

We say that a basis $\{e_i\}$ is {\bf compatible} with the filtration, if the action of a general element $c \in C$ is given by
\begin{equation}
\label{eq:compatible}
\mathfrak{a}(r_1e_1+\ldots+r_ne_n) = \max\{\mathfrak{a}(e_i); \: r_1 \neq 0  \}, \:\:\: r_i \in \kk.
\end{equation}
Such a bases always exist for any filtered complex by a result due to Barannikov \cite{Barannikov}; also see \cite[Lemma 2.2]{Dimitroglou:Persistence}. (For a general basis one would have to replace the equality "$=$" in Formula \eqref{eq:compatible} with an \emph{inequality} "$\le$".)

Given a basis with a specified action on each basis element, one can also use the above formula to \emph{construct} a filtration on the entire complex, under the assumption that the differential decreases action. The Floer complexes described below get endowed with filtrations in precisely this manner, i.e.~by specifying an action for each canonical and geometrically induced basis element.

To every complex of vector spaces equipped with a filtration there is a notion of a {\bf barcode}; we refer to \cite[Section 2]{Dimitroglou:Persistence} for the details of the presentation that we rely on here. The barcode is a set of intervals of the form $[a,b)$ and $[a,+\infty)$, where $a,b\in \RR$, where we allow multiplicities. Instead of giving the usual definition of the barcode, we give it the following alternative characterisation.
\begin{lma}[Lemma 2.6 in \cite{Dimitroglou:Persistence}]
The barcode can be recovered from the following data:
\begin{enumerate}
\item For any basis which is compatible with the action filtration, there is a bijection between the set of actions of basis elements and the union of start and endpoints of bars (counted with multiplicities).
\item For any two numbers $a<b$, the number of bars of $C_*$ whose endpoints $e$ satisfy $e \in (b,+\infty]$ and starting points $s$ satisfy $s \in [a,b)$ is equal to $\dim H(C^{<b}/C^{<a})$.
\end{enumerate}
\end{lma}
\begin{cor}
\label{cor:depth}
Assume that the barcode contains a finite bar $[a,b)$. Then, for any compatible basis $\{e_i\}$, we can deduce the existence of basis elements $e_i$ and $e_j$ with $\mathfrak{a}(e_i)=b$, $\mathfrak{a}(e_j)=b$, and for which $\langle \partial e_i,e_j \rangle \neq 0$.

Conversely, if there exists a compatible basis $\{e_i\}$ for which $\partial e_i = r e_j$ for some coefficient $r \neq 0$, then the barcode contains the finite bar $[\mathfrak{a}(e_j),\mathfrak{a}(e_i))$.
\end{cor}
\begin{rmk}
It is important that the barcode considered here does not depend on the grading in any way. An efficient way to deduce properties of the barcode is nonetheless to find a grading for the compatible basis which makes the differential an operation of degree $-1$. This imposes restrictions on the differential, which in view of the previous corollary imposes restrictions on the barcode. This technique will be crucial when studying our examples.
\end{rmk}

For a filtered complex as above we can associate the following important notions.
\begin{dfn}
\label{dfn:main}
\begin{enumerate}
\item The {\em spectral range} $\rho(C,\partial,\mathfrak{a}) \in \{-\infty\} \cup [0,+\infty]$ is the supremum of the distances between starting points of the semi-infinite bars in the barcode.
\item The {\em boundary depth} $\beta(C,\partial,\mathfrak{a}) \in \{-\infty\} \cup [0,+\infty]$ is supremum of the lengths of the finite bars in the barcode.
\end{enumerate}
\end{dfn}

An important feature of the barcode is that remains invariant under simple bifurcations of the complex, i.e.~action preserving handle-slides and birth/deaths. Legendrian isotopies induce one-parameter families of the Floer complex considered here, which undergoes bifurcations of precisely this type; hence the corresponding barcode undergoes continuous deformations under Legendrian isotopies. Since this property will not be needed, we do not give more details here, but instead direct the interested reader to \cite{Dimitroglou:Persistence}.

\subsection{Floer theory in the setting of exact Lagrangians and Legendrians}
\label{sec:floer}

Floer homology for pairs $(L_0,L_1)$ of closed exact Lagrangian submanifolds of cotangent bundles were originally defined by Floer \cite{Floer:Morse}. For any such pair one obtains the Floer chain complex $CF(L_0,L_1)$ with a basis given by the intersections $L_0 \cap L_1$, which here are assumed to be transverse. Floer also showed that the homology of the complex -- the so-called Floer homology $HF(L_0,L_1)$ -- is invariant under Hamiltonian isotopy of either Lagrangian $L_i$. Moreover, in the case when $L_1$ is a $C^1$-small Hamiltonian perturbation of $L_0$ the Floer complex $CF(L_0,L_1)=C^{\mathrm{Morse}}(f)$ is the Morse complex for a $C^1$-small Morse function $f \colon L_0 \to \RR$. (This is no longer true for the Floer homology of Legendrians; see Section \ref{sec:why}.)

Nowadays there are several different techniques available for constructing Floer homology. Here we will consider the setting of Legendrian submanifolds of contactisations $(\overline{W} \times \RR,\alpha_{st})$ of a Liouville manifold $(\overline{W},d\lambda)$, in which Floer homology associates a chain complex $CF(\Lambda_0,\Lambda_1)$ to a pair of Legendrian submanifolds equipped with additional data. In this case, the homology of the complex is invariant under Legendrian isotopy of either Legendrian $\Lambda_i$. This is the version that we will use also in the case of exact Lagrangian embeddings in $(\overline{W},d\lambda).$ To that end, recall that exact Lagrangians admit lifts to Legendrians by Lemma \ref{lma:laglift}, and that a Hamiltonian of the Lagrangian induces a Legendrian isotopy of the Legendrian lift by Lemma \ref{lma:lift}.

In the case when $\overline{W}=T^*M$, and thus $\overline{W} \times \RR =J^1M$, in  \cite{Zapolsky:Jet} Zapolsky relied on Floer homology defined using the theory of generating families due to Chekanov \cite{Chekanov:Generating} in order to define spectral invariants. Since we will work with contactisations that are more general than jet-spaces, we instead follow the techniques from \cite{DualityLeg} by Ekholm--Etnyre--Sabloff, where the Floer chain complex is constructed as the linearised Legendrian contact-homology complex associated to the Chekanov--Eliashberg algebra \cite{DiffAlg}, \cite{ContHomP}.

First we outline the general set-up Floer homology in this setting, which applies equally well to either the version used here or the version defined by using generating families (when applicable). Given a pair of Legendrians $\Lambda_0,\Lambda_1\subset \overline{W} \times \RR$, equipped with additional data denoted by $\varepsilon_i$ to be specified below (in the version defined using generating families, this additional data is simply the choice of a generating family), one obtains a graded (grading is in $\ZZ$ or $\ZZ/\mu\ZZ$ depending on the Maslov class as described in Section \ref{sec:Maslov}) filtered chain complex
$$(CF_*((\Lambda_0,\varepsilon_0),(\Lambda_1,\varepsilon_1)),\partial,\mathfrak{a})$$
with a canonical compatible basis as a $\kk$-vector space given by the
\begin{itemize}
\item Reeb chords $c$ from $\Lambda_0$ to $\Lambda_1$ of action $\mathfrak{a}(c)=\ell(c)$ equal to the Reeb chord length; together with the
\item Reeb chords $c$ from $\Lambda_1$ to $\Lambda_0$ of action $\mathfrak{a}(c)=-\ell(c)$ equal to minus the Reeb chord length.
\end{itemize}
We assume that all Reeb chords are transversely cut out, and hence that they form a discrete subset, which thus is finite whenever the Legendrians are closed. With our conventions the differential is \emph{strictly action decreasing and of degree $-1$.}

The Floer complex satisfies the following important properties; see \cite{DualityLeg} for details.
\begin{itemize}
\item A Legendrian isotopy of the Legendrian $\Lambda_i$ induces a canonical continuation of the additional data $\varepsilon_i$, and the resulting one-parameter family of Floer complexes undergoes only simple bifurcations, i.e.~handle-slides and births/deaths. In particular, the homology of the complex is not changed under such a deformation.
\item In the case when $\Lambda \subset \overline{W} \times \RR$ has no Reeb chords (i.e.~it is the lift of an exact Lagrangian embedding), and when $\Lambda'$ is a $C^1$-small Legendrian perturbation, then the induced Floer complex
$$(CF((\Lambda,\varepsilon),(\Lambda',\varepsilon')),\partial,\mathfrak{a})=C^{\mathrm{Morse}}(f;\kk)$$
is the Morse homology complex of some $C^1$-small Morse function $f \colon \Lambda \to \RR$.
\end{itemize}
Again we refer to Section \ref{sec:why} for a description of the complex under the presence of pure Reeb chords; in this case the Morse complex is only realised as a quotient complex of a subcomplex.

\subsubsection{Floer complex as the linearised Chekanov--Eliasbherg algebra}
Here we relevant technical details for the particular construction of Floer homology used here, i.e.~relying on the Chekanov--Eliashberg algebra for Legendrians in contactisation from \cite{ContHomP}.

Assume that $\Lambda_0,\Lambda_1 \subset \overline{W} \times \RR$ are two Legendrian submanifolds. Further, assume that the Chekanov--Eliashberg algebras of $\Lambda_i$ admit augmentations
$$\varepsilon_i \colon (\mathcal{A}(\Lambda_i),\partial) \to \kk;$$
recall that the Chekanov--Eliasbherg algebra is a unital DGA generated by the Reeb chords of the Legendrian, and that an augmentation is a unital DGA morphism to the ground field. In particular, when the Legendrian $\Lambda_i$ has no Reeb chords, the Chekanov--Eliashberg algebra takes the simple form $\mathcal{A}(\Lambda_i)=\kk,$ and there is a canonical augmentation. An important property of augmentations is that they can be pushed forward under a Legendrian isotopy; see e.g.~\cite{DiffAlg} and \cite{Dimitroglou:Cthulhu}.

Typically one wants more additional data than just an augmentation. For instance, in order to use coefficients in a field of characteristic different from two, one also needs to fix the choice of a spin structure on both Legendrians $\Lambda_i$. In order to endow the Floer complex a $\ZZ$-grading, we need to specify a Maslov potential; we refer to Subsection \ref{sec:Maslov} for more details concerning the grading, which will play an important role for us.

The Floer complex
$$ CF((\Lambda_0,\varepsilon_0),(\Lambda_1,\varepsilon_1)) $$
is generated by the chords that have one endpoint on $\Lambda_0$ and one endpoint on $\Lambda_1$ (either being a starting point). These Reeb chords on $\Lambda_0 \cup \Lambda_1$ are called the {\bf mixed} Reeb chords. In order to define the differential, we will identify the above vector space with the underlying vector space linearised Legendrian contact homology complex of the link $\Lambda_0 \cup \phi^T_{\partial_z}(\Lambda_1)$, where the latter is the $\kk$-vector space is generated by all Reeb chords that start on $\Lambda_0$ and end on the translation $\phi^T_{\partial_z}(\Lambda_1)$ of $\Lambda_1$ in the positive $z$-direction. Note that the mixed chords on $\Lambda_0 \cup \Lambda_1$ are in bijective correspondence with the mixed chords on $\Lambda_0 \cup \phi^T_{\partial_z}(\Lambda_1)$. Here we require that $T \gg 0$ has been chosen sufficiently large, so that no chord starts on $\phi^T_{\partial_z}(\Lambda_1)$ and ends on $\Lambda_0$. Of course, the length of a mixed chord $c$ above depend on the parameter $T$ and will not be equal to the action $\mathfrak{a}(c)$ defined above; the relation between action and length is precisely
$$\ell(c)=\mathfrak{a}(c)+T.$$
The remaining Reeb chords on the link $\Lambda_0 \cup \phi^T_{\partial_z}(\Lambda_1)$ have both endpoints either on $\Lambda_0$ or $\phi^T_{\partial_z}(\Lambda_1)$, and are called {\bf pure}. Note that the Reeb chords on $\phi^T_{\partial_z}(\Lambda_1)$ are in bijective correspondence with those of $\Lambda_1$. In fact, their Chekanov--Eliashberg algebras are even canonically isomorphic.

The differential is the Linearised Legendrian contact homology differential induced by a choice of almost complex structure, together with the augmentations $\varepsilon_i$ for the Chekanov--Eliashberg algebras $\mathcal{A}(\Lambda_i)$ generated by the pure chords. This version of a Floer complex defined via the Chekanov--Eliashberg algebra was originally considered in \cite{DualityLeg}; also see \cite{Dimitroglou:Cthulhu} for a more recent realisation. We now give a sketch of the definition of the differential. It is roughly speaking defined by counts of rigid pseudoholomorphic discs in $\overline{W}$ with
\begin{itemize}
\item boundary on the Lagrangian immersion $\Pi(\Lambda_0 \cup \phi^T_{\partial_z}(\Lambda_1)) \subset \overline{W}$;
\item precisely one positive puncture at a double point which corresponds to a mixed chord -- this is the input;
\item precisely one negative puncture at a double point which corresponds to a mixed chord -- this is the output; and
\item several additional negative punctures at double points which correspond to pure chords.
\end{itemize}
When counting the strip, one weights the count by the value of the augmentation $\varepsilon_i$ on the latter pure chords. This is a part of the so-called linearised differential induced by the augmentation, as defined in \cite{DiffAlg}; also see the notion of the bilinearised Legendrian contact homology as defined by Bourgeois--Chantraine in \cite{Bilinearised}.

From positivity of symplectic area of such pseudoholomorphic discs together with Stokes' theorem one obtains that the Reeb chord length of the input chord must be larger than the Reeb chord of the output. In other words, the complex is filtered in the precise sense defined in Section \ref{sec:barcode}.

From the index formula for the expected dimension of a pseudoholomorphic discs, it follows that the degree of the input is one greater than the degree of the output; i.e.~the differential is of degree $-1$.

\subsubsection{Maslov potential and grading}
\label{sec:Maslov}

In order to define grading in Lagrangian Floer homology the technique of Maslov potentials is useful. The construction of a Maslov potential is originally is due to Seidel \cite{Seidel:Graded}. If the Maslov potential is defined then the grading is well-defined in $\ZZ$, in general the potential is only well-defined modulo the Maslov number $\mu \in \ZZ$ (the generator of the subgroup of $\ZZ$ which is the image of the Maslov class) and in that case the grading is only defined in $\ZZ/\mu\ZZ$. In any case the differential is of degree $-1$ with our conventions (i.e.~it decreases the grading).

Assume that $\overline{W}$ has vanishing first Chern class; this is e.g.~the case when $\overline{W}$ has a symplectic trivialisation, which is automatic when $\dim_{\RR} \overline{W}=2$. The $\ZZ$--grading of the generators is defined as follows. 

First, one makes the choice of a trivialisation of the determinant bundle
$$\CC^* \to \det_{\CC} T\overline{W} \to \overline{W}$$
induced by some choice of a compatible almost complex structure. There is an induced bundle with fibre
$$\CC^*/{\sim}=(\RR^2\setminus \{0\})/{\sim}=\RR P^1 \cong S^1$$
which admits the lift to an affine $\RR$-bundle via the universal cover $\RR \to S^1.$

Second, one makes the choice of a Maslov potential for each of the Legendrians $\Lambda_i$; this is the lift of the canonically defined section $\det_{\RR}T\Pi(\Lambda_i)/{\sim}$ inside the above $S^1$-bundle to the associated $\RR$-bundle. Recall the a non-zero Maslov class is the obstruction to the existence of such a lift. When a Maslov potential exists and the Legendrian is connected, two different choices of Maslov potentials differ by the addition of an integer.

Finally, the grading of a generator $c \in CF_*((\Lambda_0,\varepsilon_0),(\Lambda_1,\varepsilon_1))$ is obtained in the following manner. Consider the path of Lagrangian planes given by rotating the Lagrangian plane $T_{\Pi(c)}\Pi(\Lambda_0) \subset T_{\Pi(c)}\overline{W}$ to $T_{\Pi{c}}\Pi(\Lambda_1) \subset T_{\Pi(c)}\overline{W}$ through the smallest possible positive K\"{a}hler angles. These rotations induce a continuous deformation of the Maslov potential of $\Lambda_0$ at the point $\Pi(c)$; denote by $\mu_0 \in \RR$ the new value. By construction, the deformed Maslov potential of $\Lambda_0$ at $\Pi(c)$ and the Maslov potential of $\Lambda_1$ at $\Pi(c)$ are now lifts to $\RR$ of the same point in $S^1$. The grading is the number $\mu_0-\mu_1 \in \ZZ$ which is integer by the last property.

\begin{lma}
\label{lma:Maslov}
\begin{enumerate}
\item Let $\phi^1 \colon W \times \RR \to W \times \RR$ be the time-one map of a compactly supported contact isotopy. For any choice of Maslov potential on the Legendrian $\Lambda$ there an induced Maslov potential on its image $\phi^1(\Lambda) \subset W \times \RR$ uniquely defined by the property that the Maslov potentials extend over the exact Lagrangian cobordism from $\Lambda$ to $\phi^1(\Lambda)$ induced by the isotopy.
\item If $\phi^1$ is a generic $C^1$-small contact isotopy, then the small chords of $\Lambda \cup \phi^1(\Lambda)$ are in bijective correspondence with the critical points of a $C^1$-small Morse function $f \colon \Lambda \to \RR$, and the above grading coincides with the Morse index.
\end{enumerate}
\end{lma}
\begin{proof}
(1) The trace of the Legendrian isotopy can be made into a Lagrangian cylinder inside the symplectisation $$(\RR_t \times \overline{W} \times \RR_z,d(e^t\alpha_{st}))$$
with cylindrical ends over the initial and final Legendrian; see work \cite{LagrConc} by Chantraine. The Maslov potential of $\Lambda$ induces a Maslov potential on the negative end of this cobordism. This Maslov potential can be extended to the entire cobordism by elementary topology (it is a Lagrangian cylinder). The induced Maslov potential on the positive end is the sought Maslov potential on $\phi^1(\Lambda)$.

(2) This computation is standard, and can be performed in a small neighbourhood of $\Lambda$. Recall that any Legendrian $\Lambda$ has a standard neighbourhood which is contactomorphic to a neighbourhood of the zero section $j^10 \subset J^1\Lambda$, under which $\Lambda$ moreover is identified with $j^10$; see \cite{Geiges:Intro}. The perturbation can be assumed to be given by the one-jet $j^1f$ of some $C^1$-small smooth function $f \colon \Lambda \to \RR$ in the same neighbourhood.
\end{proof}

Note that the case when $\overline{W}$ is of dimension $\dim_{\RR}\overline{W}=2$ then $T\overline{W}$ always has a symplectic trivialisation.

In the case when $\overline{W}=T^*S^1$ there is a canonically defined trivialisation in which the zero-section has a constant field of non-zero tangent vectors. With this trivialisation the zero section obviously has a Maslov potential which moreover is constant (for a suitable trivialisation). The different symplectic trivialisations on $T^*S^1$ up to homotopy are in bijection with homotopy classes of maps $T^*S^1 \to S^1$ i.e.~cohomology classes in $H^1(T^*S^1)$. Note that there is a unique trivialisation for which the zero section has a non-vanishing Maslov class; for the remaining trivialisations the zero-section does not admit a Maslov potential.

\section{Examples that exhibit unbounded spectral norms}

The following auxiliary result facilitates our computations, and will be invoked repeatedly.
\begin{lma}
\label{lma:spectralnormcomp}
\begin{enumerate}
\item Let $\phi^t \colon \Lambda_0 \hookrightarrow \overline{W} \times \RR$ be a Legendrian isotopy of a closed Legendrian $\Lambda_0$ that admits a Maslov potential, and endow $\phi^1(\Lambda_0)$ with the Maslov potential induced from $\Lambda_0$ via the isotopy, as described in Part (1) of Lemma \ref{lma:Maslov}. Further assume that $\Lambda_0$ has no Reeb chords. If the complex $CF(\Lambda_0,\phi^1(\Lambda_0))$ has unique Reeb chord generators $c$ and $d$ in degrees $|d|=0$ and $|c|=\dim \Lambda_0$, then the spectral range satisfies
$$ \rho(CF(\Lambda_0,\phi^1(\Lambda_0))) \ge |\ell(c)-\ell(d)|.$$
(In fact, it is even true that the spectral range is \emph{equal} to $\ell(c)-\ell(d)$, where this quantity moreover is positive, but we will not show this.)
\item Assume that the complex $CF(\Lambda_0,\Lambda_1)$ is $\ZZ$-graded, acyclic, and has no generators in degrees $i+1$ or $i-2.$ If there are unique Reeb chords $c,d$ in the degrees $|c|=i$ and $|d|=i-1$, for some choice of symplectic trivialisation and Maslov potential, then the boundary depth satisfies the bound
$$\beta(CF(\Lambda_0,\Lambda_1)) \ge \ell(c)-\ell(d).$$
\end{enumerate}
\end{lma}
\begin{proof}
(1): This follows from invariance properties of the Floer homology. Note that the homology of $CF(\Lambda_0,\Lambda_0)$ has unique generators in degrees 0 and $\dim\Lambda$ which represent the point class and fundamental class in Morse homology. It follows by degree reasons that the Reeb chord generators $c$ and $d$ must both be cycles which are not boundaries. The two corresponding semi-infinite bars in the barcode have endpoints that are separated by precisely $|\ell(c)-\ell(d)|$ as sought.

(2): Acyclicity together with the degree assumptions implies that $\partial c=d$. The statement then follows by the second part of Corollary \ref{cor:depth} since the Reeb chords form a compatible basis.
\end{proof}

\subsection{Legendrian isotopy of the unknot (Proof of Part (2) of Theorem \ref{mainthm})}
\label{sec:unknot}

Consider the contact manifold $J^1\RR=\RR_q \times \RR_p \times \RR_z$ with coordinates $q,p,z$ and contact form $dz-p\,dq$. Under the quotient $\RR_q \to \RR/\ZZ=S^1$ we obtain the angular coordinate $\theta$ induced by $\theta=q \mod 1$. In other words, the aforementioned contact manifold $J^1\RR$ is the universal cover of the contact manifold $J^1S^1=S^1 \times \RR_p \times \RR_z$ equipped with the standard contact form $dz-p\,d\theta$.

First consider the standard Legendrian unknot $\Lambda_0 \subset J^1S^1$ with front projection as shown in Figure \ref{fig:u1}, which thus is contained inside the subset $J^1[-1/2,1/2] \subset J^1S^1$. The $p$-coordinate of this particular representative can be seen to be estimated in terms of the ratio of $a$ and $b$, which yields $$\Lambda_0 \subset \{|p| \le 2a/b\}.$$

Recall the well-known fact that $\Lambda_0$ has vanishing Maslov class and hence admits a Maslov potential. Further, this Legendrian has a unique transverse Reeb chord and its Chekanov--Eliashberg algebra is equal to the polynomial algebra in one variable of degree $1$ with no differential (either for $\kk=\ZZ_2$ or for arbitrary $\kk$ and the choice of bounding spin structure); see \cite{Etnyre:LegendrianContact}. In particular, its Chekanov--Eliashberg algebra admits the trivial augmentation.

We also fix a Legendrian fibre
$$F=F_{(1/4,0)}=\{(1/4) \times \RR_p \times \{0\}\} \subset J^1[-1/2,1/2] \subset J^1S^1.$$
Note that the Reeb chords between any Legendrian $\Lambda$ and $F$ are in bijective correspondence with the intersection points of $\Lambda$ and the hypersurface $\{\theta=1/4\}$.

Since $F$ that has no Reeb chords, its Chekanov--Eliashberg algebra trivially admits an augmentation. We can thus define the Floer homology complex $CF(\Lambda_0,F)$ which is generated by two Reeb chords $c$ and $d$, where $0>\ell(c)>\ell(d)$ and $|c|=|d|+1$. Note that $CF(\Lambda_0,F)$ is an acyclic complex by invariance under Legendrian isotopy; after shrinking the unknot sufficiently, all mixed chords disappear.

The goal is to construct a Legendrian isotopy $\Lambda_t \subset J^1S^1$ of the unknot confined to the subset
$$\{|p| \le 2a/b\} \subset J^1 S^1$$
for which the boundary depth of $CF(\Lambda_T,F)$ becomes arbitrarily large as $t \to +\infty$. This isotopy will be constructed as the projection of an isotopy $\tilde{\Lambda}_t \subset J^1\RR$ of the unknot inside the universal cover $J^1\RR \to J^1S^1$. In fact, the Legendrian isotopy $\tilde{\Lambda}_t$ is very simple; it is the rescaling of
$$\tilde{\Lambda}_0=\Lambda_0 \subset J^1[-1/2,1/2] \subset J^1\RR$$
under the map $(q,p,z) \mapsto (e^t\cdot q,p,e^t\cdot z)$.

It is easy to check that $CF(\tilde{\Lambda}_t,F)$ satisfies the property that the boundary depth goes to $+\infty$ as $t \to +\infty$. Indeed, these complexes are generated by the two unique transversely cut out Reeb chords $c_t$ and $d_t$ between $\tilde{\Lambda}_t$ and $F$ for all values $t>0$. These chords moreover satisfy the property that  $\ell(c_t) - \ell(d_t)$ becomes arbitrarily large as $t \to +\infty$; c.f.~Part (2) of Lemma \ref{lma:spectralnormcomp}. 

What remains to prove is the following two claims for the projection $\Lambda_t \subset J^1S^1$ of the Legendrian rescaling $\tilde{\Lambda}_t \subset J^1\RR$. First, we claim that $\Lambda_t$ indeed is a Legendrian isotopy. Second, we show that the boundary depth of $CF(\Lambda_t,F)$ goes to $+\infty$ as $t \to +\infty$

The fact that $\Lambda_t$ is a Legendrian isotopy can be seen by considering the sequence of front projections; see Figures \ref{fig:u2} and \ref{fig:u3}. Except for an isotopy of the front, the front also undergoes a sequence \emph{RIII}-moves together with the composite move shown in Figure \ref{fig:r}.

Then we need to estimate the boundary depth of the sequence of Floer complexes $CF(\Lambda_t,F)$. In addition to Reeb chords $c_t$ and $d_t$, which correspond to the Reeb mixed Reeb chords on the lift and have exactly the same actions, there are additional Reeb chords between $\Lambda_t$ and $F$ that appear as $t \to +\infty$. Nevertheless, we claim that the boundary depth of $CF(\Lambda_t,F)$ still is bounded from below by the boundary depth $\beta(CF(\tilde{\Lambda}_t,F)).$

To see the last claim, we will consider different gradings of the complexes $CF(\Lambda_t,F)$, obtained by changing the symplectic trivialisation of $T^*S^1$. Note that $\Lambda_0$ is null-homotopic inside $J^1S^1$ and thus has a vanishing Maslov class independently of the choice of symplectic trivialisation. Moreover, the chords $c_t$ and $d_t$ always satisfy $|c_t|-|d_t|=1$ regardless of the choice of Maslov potential and symplectic trivialisation.

We claim that, after changing the symplectic trivialisation of $T^*S^1$ by introducing a sufficiently large number $N \gg 0$ of rotations of the standard symplectic frame as one traverses $\theta=1/2$, all generators $c'$ in the complex except $c_t$ and $d_t$ acquire degrees that satisfy
$$|c'|- |c_t| \notin [-10,10].$$
Since these degree properties can be achieved, the statement now follows directly by Part (2) of Lemma \ref{lma:spectralnormcomp}.\qed

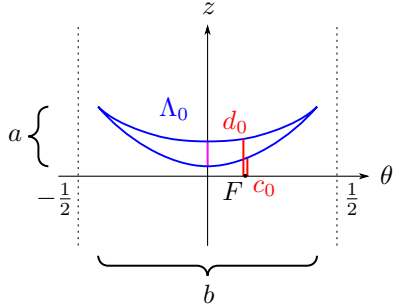
\begin{figure}[htp]
	\vspace{3mm}
	\labellist
	\pinlabel $a$ at -5 53
	\pinlabel $b$ at 68 -7
	\pinlabel $\color{blue}\Lambda_0$ at 55 63
	\pinlabel $\color{red}d_0$ at 78 58
	\pinlabel $\color{red}c_0$ at 89 33
	\pinlabel $F$ at 77 32
	\pinlabel $z$ at 68 101
	\pinlabel $\theta$ at 135 38
	\pinlabel $\frac{1}{2}$ at 122 29
	\pinlabel $-\frac{1}{2}$ at 10 29
	\endlabellist
	\includegraphics{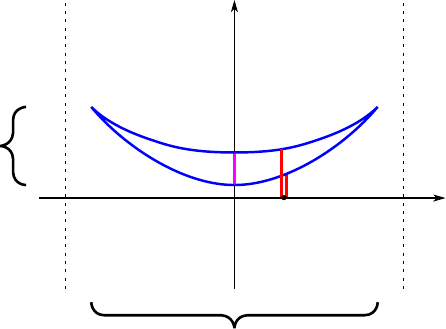}
	\vspace{3mm}
	\caption{The standard Legendrian unknot $\Lambda_0$ and the Legendrian fibre $F$. Note that there are precisely two transverse Reeb chords $c_0,d_0$ between $F$ and $\Lambda_0$.}
	\label{fig:u1}
\end{figure}
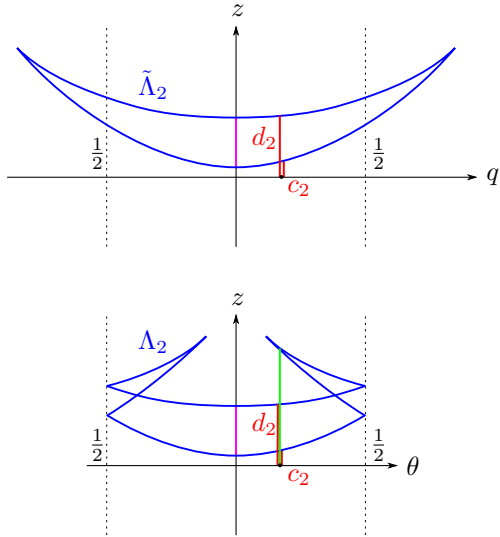
\begin{figure}[htp]
	\vspace{3mm}
	\labellist
	\pinlabel $\color{blue}\Lambda_2$ at 55 73
	\pinlabel $\color{blue}\tilde{\Lambda}_2$ at 55 170
	\pinlabel $\color{red}d_2$ at 97 42
	\pinlabel $\color{red}c_2$ at 110 22
	\pinlabel $\color{red}d_2$ at 97 150
	\pinlabel $\color{red}c_2$ at 110 131
	\pinlabel $z$ at 87 89
	\pinlabel $z$ at 87 198
	\pinlabel $\theta$ at 153 26
	\pinlabel $q$ at 183 135
	\pinlabel $\frac{1}{2}$ at 140 144
	\pinlabel $\frac{1}{2}$ at 34 144
	\pinlabel $\frac{1}{2}$ at 140 34
	\pinlabel $\frac{1}{2}$ at 34 34
	\endlabellist
	\includegraphics{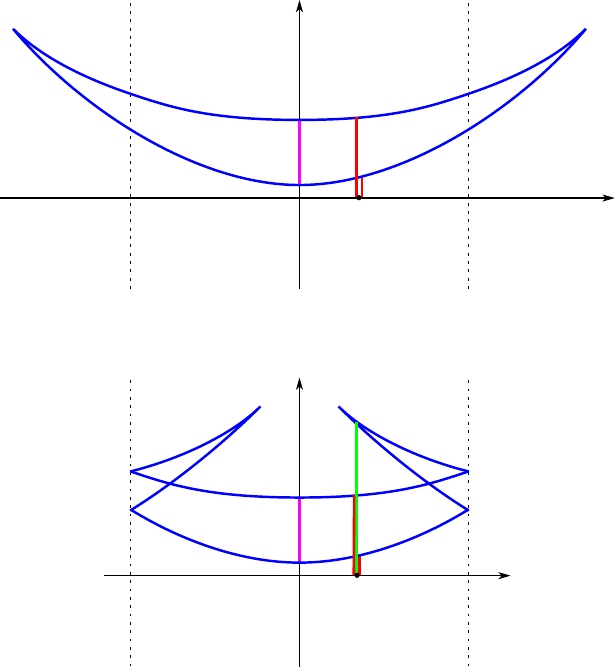}
	\caption{Above: $\tilde{\Lambda}_2$ has a front which is a linear rescaling of the front of $\Lambda_0$ inside $J^1\RR$. Below: $\Lambda_2$ is the projection of $\tilde{\Lambda}_2$ inside $J^1S^1$. Except for the mixed chords $c_t$ and $d_t$ that exist for the lift, there are now additional mixed chords.}
	\label{fig:u2}
\end{figure}
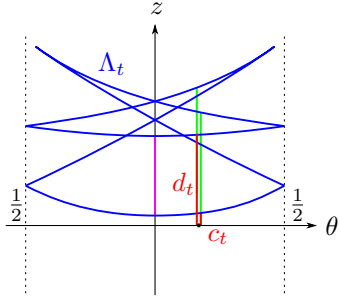
\begin{figure}[htp]
	\vspace{3mm}
	\labellist
	\pinlabel $\color{blue}\Lambda_t$ at 40 87
	\pinlabel $\color{red}d_t$ at 67 42
	\pinlabel $\color{red}c_t$ at 80 22
	\pinlabel $z$ at 57 108
	\pinlabel $\theta$ at 123 26
	\pinlabel $\frac{1}{2}$ at 110 34
	\pinlabel $\frac{1}{2}$ at 4 34
	\endlabellist
	\includegraphics{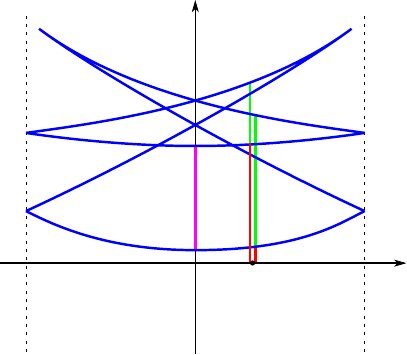}
	\caption{This shows the projection $\Lambda_t$ of the rescaling $\tilde{\Lambda}_t$ under the universal cover $J^1\RR \to J^1S^1$.}
	\label{fig:u3}
\end{figure}

\subsection{Legendrian isotopy of the zero-section (Proof of Part (1) of Theorem \ref{mainthm})}
We use the same coordinates as in the above Section \ref{sec:unknot}. In fact, the sought Legendrian isotopy is also constructed in a manner similar to the construction of $\Lambda_t$ given there, by performing a rescaling of a part of the front inside the universal cover $J^1\RR$ (and then projecting back to $J^1S^1$). The isotopy is shown in Figures \ref{fig:d1} and \ref{fig:d4}. One starts by considering a Legendrian perturbation $j^1f$ of $j^10$ which has precisely two chords. Then one performs a \emph{RII}-move. Rescaling the front of the Legendrian introduced by the \emph{RII}-move in the universal cover $\RR^2$ and then projecting back to $S^1 \times \RR$ is again a Legendrian isotopy. In Figure \ref{fig:d4} one sees that there are exactly two chords between $j^10$ and the produced Legendrians, while the difference in action between these two generators grows indefinitely as $t \to +\infty$. \qed

\begin{figure}[htp]
	\vspace{3mm}
	\labellist
	\pinlabel $z$ at 56 90
	\pinlabel $\theta$ at 122 27
	\pinlabel $z$ at 203 90
	\pinlabel $\theta$ at 269 27
	\pinlabel $\frac{1}{2}$ at 110 35
	\pinlabel $-\frac{1}{2}$ at -1 35
	\pinlabel $\frac{1}{2}$ at 257 35
	\pinlabel $-\frac{1}{2}$ at 146 35
	\pinlabel $\color{red}d$ at 3 17
	\pinlabel $\color{red}d$ at 101 17
	\pinlabel $\color{red}c$ at 52 32
	\pinlabel $\color{red}d$ at 150 17
	\pinlabel $\color{red}d$ at 248 17
	\pinlabel $\color{red}c_0$ at 197 40
	\pinlabel $\color{blue}\Lambda$ at 35 7
	\pinlabel $\color{blue}\Lambda_0$ at 185 7
	\endlabellist
	\includegraphics{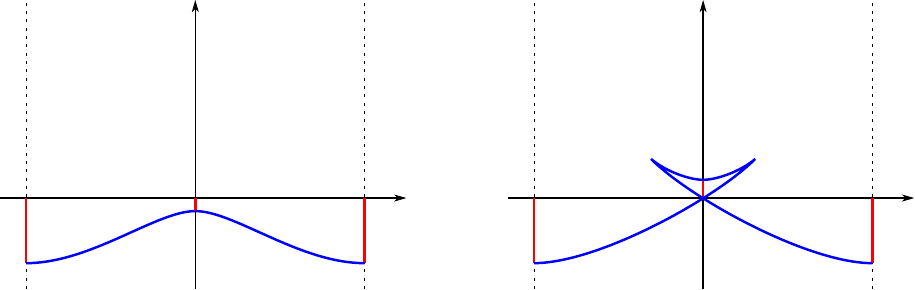}
	\caption{Left: A Legendrian perturbation of the zero section. The vertical chords denote the two Reeb chords between the zero-section $j^10$ and the perturbation. Right: The perturbed version of the zero-section after a suitable Legendrian \emph{RI}-move.}
	\label{fig:d1}
\end{figure}
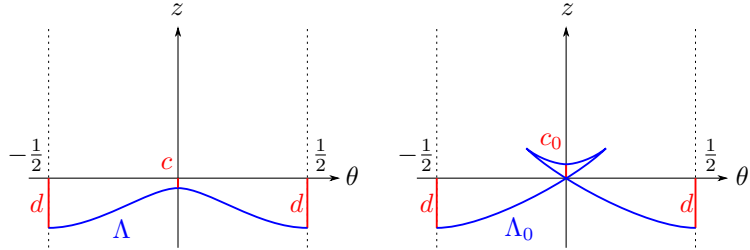

\begin{figure}[htp]
	\vspace{3mm}
	\labellist
	\pinlabel $\frac{1}{2}$ at 166 170
	\pinlabel $\frac{2}{2}$ at 215 170
	\pinlabel $-\frac{2}{2}$ at 6 170
	\pinlabel $-\frac{1}{2}$ at 56 170
	\pinlabel $z$ at 113 243
	\pinlabel $z$ at 113 115
	\pinlabel $\theta$ at 178 26
	\pinlabel $q$ at 235 161
	\pinlabel $\frac{1}{2}$ at 166 34
	\pinlabel $-\frac{1}{2}$ at 56 34
	\pinlabel $\color{red}d$ at 157 153
	\pinlabel $\color{red}d$ at 59 153
	\pinlabel $\color{red}c_t$ at 107 182
	\pinlabel $\color{red}d$ at 157 17
	\pinlabel $\color{red}d$ at 59 17
	\pinlabel $\color{red}c_t$ at 107 46
	\pinlabel $\color{blue}\Lambda_t$ at 100 10
	\endlabellist
	\includegraphics{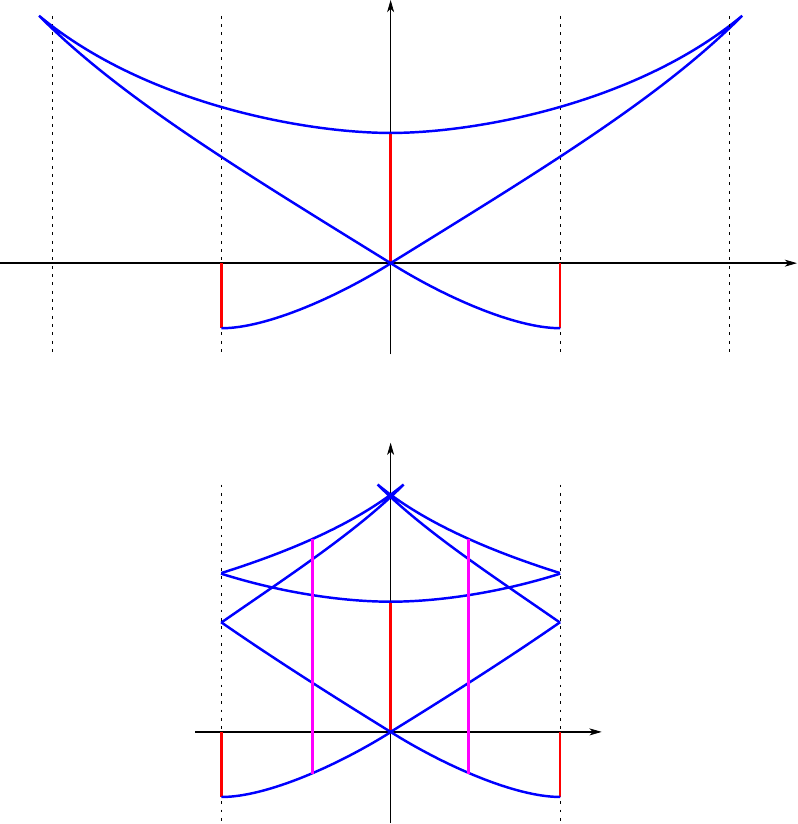}
	\caption{$\Lambda_t$ is obtained from $\Lambda_0$ by a linear rescaling of the front inside $\{ z \ge 0\}$ in the universal cover $J^1\RR^2$ followed by the canonical projection $J^1\RR \to J^1S^1.$ The front of $\Lambda_t$ undergoes the composite move shown in Figure \ref{fig:r} consisting of two consecutive \emph{RII}-moves along with \emph{RIII}-moves.}
	\label{fig:d4}
\end{figure}
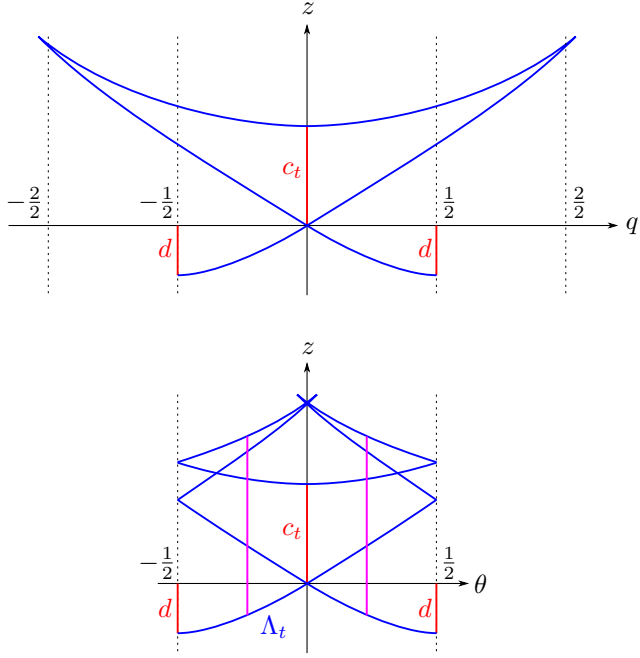

\subsection{Hamiltonian isotopy on the punctured torus (Proof of Theorem \ref{mainthm:torus})}

We consider the exact Lagrangian embedding $L \subset (\Sigma_{1,1},d\lambda)$ of $S^1$ which is given as the image of $\{p=0\} \subset \RR^2$ under the quotient construction in Section \ref{sec:torus}; see Figure \ref{fig:t1}. We perform a Hamiltonian perturbation $L'$ that intersects the original Lagrangian transversely in precisely two points $c$ and $d$. The spectral norm is thus $\gamma(CF(L,L'))=\ell(c)-\ell(d)$.

Then consider the autonomous Hamiltonian
$$\rho \colon \Sigma_{1,1} \to \RR_{\ge 0}$$
with support inside $\{q \in [-\delta,\delta]\}$ for some small $\delta>0$, and which is equal to the smooth bump-function $\rho(q) \ge 0$ in one variable of the form
\begin{itemize}
\item $\rho(q)\equiv 1$ in a neighbourhood of $q=0$;
\item $\rho(q)=\rho(-q)$;
\item and $\rho'(q) \ge 0$ for $q<0$.
\end{itemize}
The Hamiltonian isotopy $\phi^t_{\rho}$ wraps the region $q \in (-\delta,0)$ in the negative $p$-direction, while it wraps the region $q \in (0,\delta)$ in the positive $p$-direction.

We claim that $CF(L,\phi^t_\rho(L'))$ has a spectral norm which becomes arbitrarily large as $t \to +\infty$. What is clear is that $\ell(c)-\ell(d) \to +\infty$ as $t \to +\infty$. (Use e.g.~Lemma \ref{lma:lift}.) Again there are additional generators that appear as $t \to +\infty$, so knowing that $\ell(c)-\ell(d) \to +\infty$ is not sufficient.

As in Section \ref{sec:unknot} a change of symplectic trivialisation can again give us what we need. First consider the canonical symplectic trivialisation, induced by the trivialisation of $\RR^2$ and the quotient projection. Then deform this trivialisation by making a number $N \gg 0$ of full rotations of the standard symplectic frame (relative the constant one) as one traverses the cycle $\{p=1\}$. Note that the Lagrangian corresponding to $\{p=0\}$ still has a Maslov potential after this change of trivialisation. Again it is readily seen that all generators $c'$ except $c$ and $d$ satisfy the property
$$|c'|-|c| \notin [-10,10]$$
after choosing $N \gg 0$ sufficiently large, while $|c|-|d|=1$ always is satisfied.

The spectral norm can finally be computed by invoking Part (1) of Lemma \ref{lma:spectralnormcomp}.

\begin{figure}[htp]
	\vspace{3mm}
	\labellist
	\pinlabel $q$ at 182 86
	\pinlabel $p$ at 87 184
	\pinlabel $\color{red}c$ at 91 80
	\pinlabel $\color{red}d$ at 148 81
	\pinlabel $L$ at 122 80
	\pinlabel $\color{blue}L'$ at 122 102
	\pinlabel $\Sigma_{1,1}$ at 295 25
	\pinlabel $\color{blue}L'$ at 344 80
	\pinlabel $L$ at 305 70
	\endlabellist
	\includegraphics{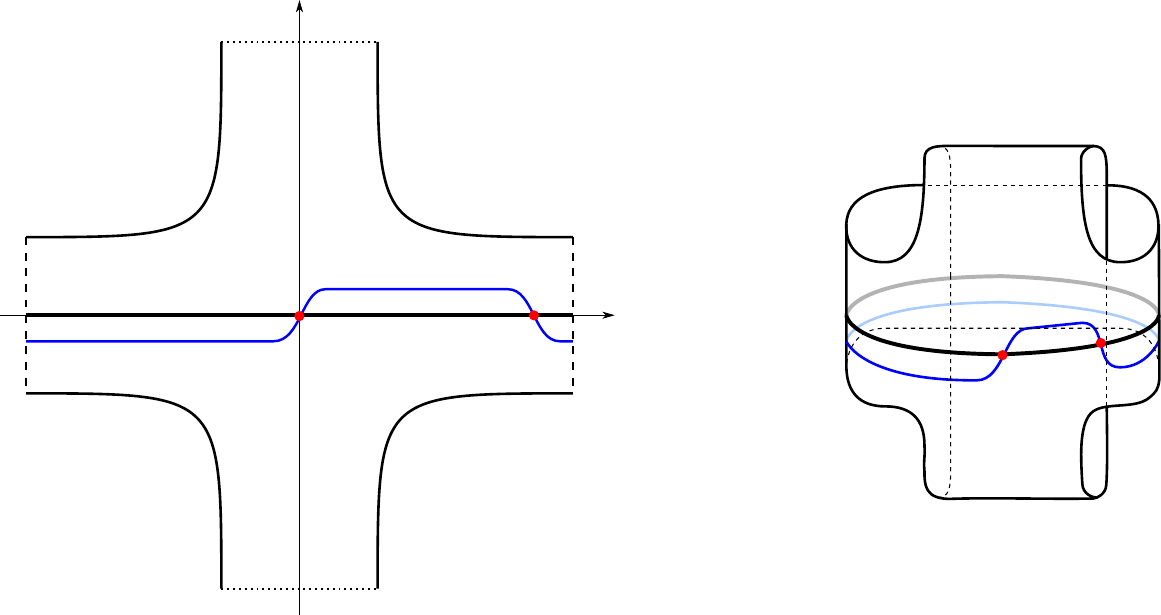}
	\caption{The left depicts a domain in $\RR^2$ with piecewise smooth boundary. After identifying the two horizontal pieces of the boundary, as well as the two vertical pieces, one obtains the Liouville domain shown on the right, with Liouville form described in Section \ref{sec:torus}. The closed exact Lagrangian $L$ is the image of $\{p=0\}$ and $L'$ is a small Hamiltonian perturbation of $L$.}
	\label{fig:t1}
\end{figure}
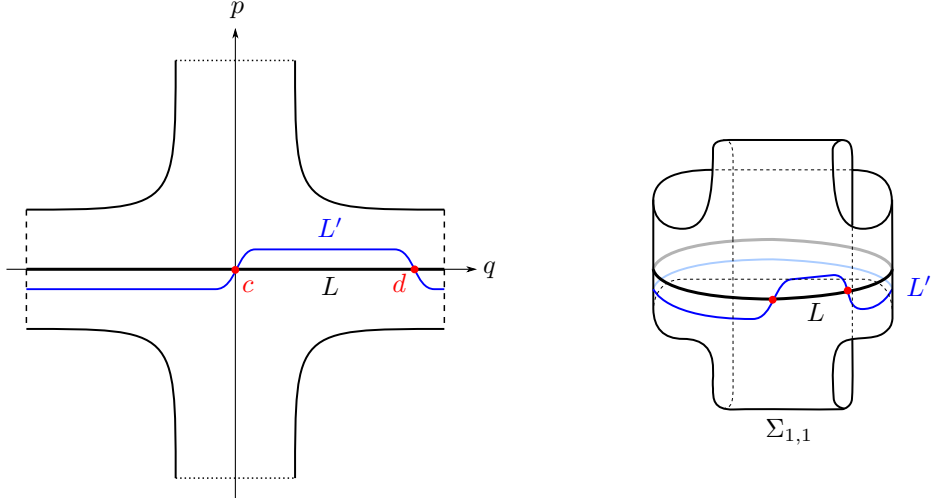

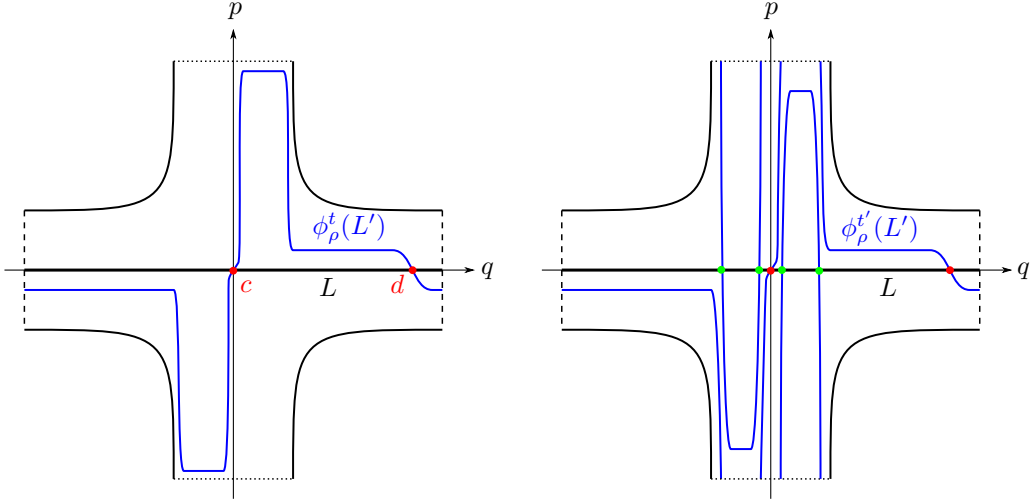
\begin{figure}[htp]
	\vspace{3mm}
	\labellist
	\pinlabel $q$ at 384 86
	\pinlabel $p$ at 289 184
	\pinlabel $q$ at 182 86
	\pinlabel $p$ at 87 184
	\pinlabel $\color{red}c$ at 91 80
	\pinlabel $\color{red}d$ at 148 81
	\pinlabel $L$ at 122 80
	\pinlabel $L$ at 333 80
	\pinlabel $\color{blue}\phi^t_{\rho}(L')$ at 130 102
	\pinlabel $\color{blue}\phi^{t'}_{\rho}(L')$ at 330 102
	\endlabellist
	\includegraphics{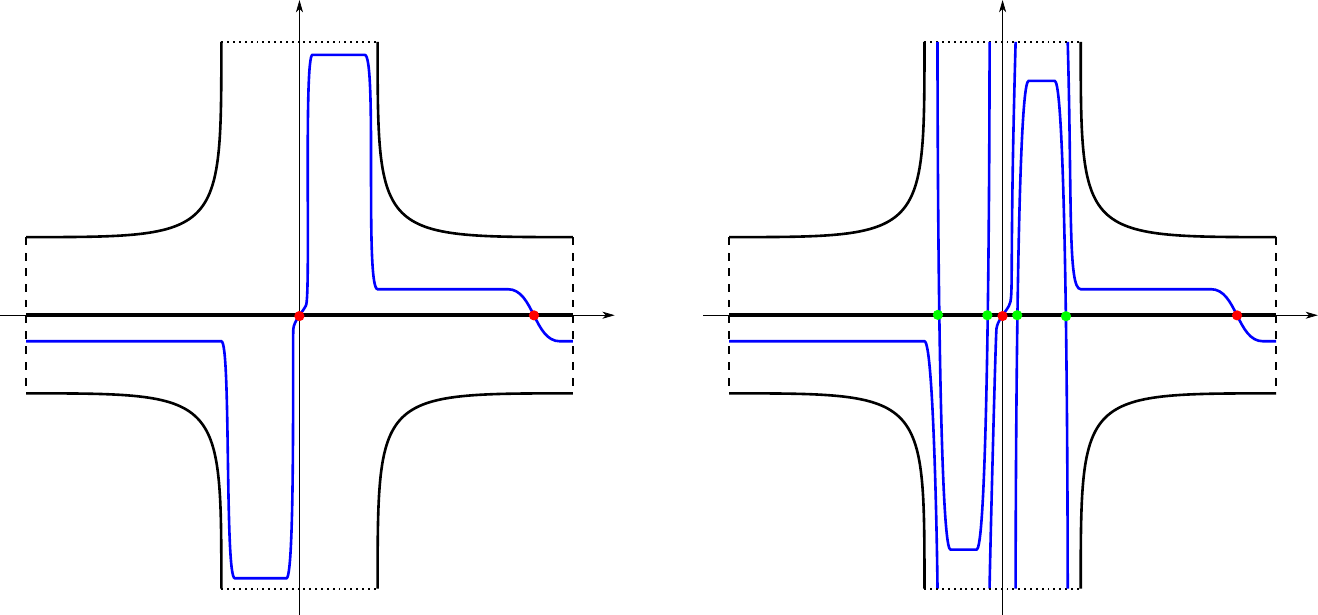}
	\caption{A Hamiltonian isotopy that wraps the Lagrangian $\{p=0\}$ around the one-handle with core $\{q=0\}$, while fixing a neighbourhood of the latter core. Note that the Hamiltonian function is positive but constant near $q=0$. Here $t'>t$.}
	\label{fig:t2}
\end{figure}

\section{Proof of Theorem \ref{thm:ambientsurgery}}

By definition, our two Floer complexes are the linearised Legendrian contact homology complexes generated as a $\kk$-vector space by the mixed Reeb chords on the Legendrian link
$$ \Lambda_\pm \cup \phi^T_{\partial_z}(\Lambda).$$
Here $T \gg 0$ is fixed but sufficiently large.

The cusp-connected sum performed on $\Lambda_- \cup \phi^T_{\partial_z}(\Lambda)$ produces $\Lambda_+ \cup \phi^T_{\partial_z}(\Lambda)$ (of course, only the first component is affected). There is an associated exact standard Lagrangian handle-attachment cobordism 
$$ \mathcal{L} \subset (\RR_t \times \overline{W} \times \RR_z,d(e^t\alpha_{st})) $$
inside the symplectisation as constructed in \cite{Dimitroglou:Ambient}. This is a cobordism with cylindrical ends from
$$\Lambda_- \cup \phi^T_{\partial_z}(\Lambda) \:\:\:\text{to}\:\:\:\Lambda_+ \cup \phi^T_{\partial_z}(\Lambda),$$
i.e.~from the Legendrian link before surgery (at the concave end) to the link after surgery (at the convex end). One component of this cobordism is simply the trivial cylinder $\RR \times \phi^T_{\partial_z}(\Lambda))$. This Lagrangian cobordism induces a unital DGA-morphism
$$\Phi_{\mathcal{L}} \colon \mathcal{A}(\Lambda_+ \cup \phi^T_{\partial_z}(\Lambda)) \to \mathcal{A}(\Lambda_- \cup \phi^T_{\partial_z}(\Lambda))$$
of the Chekanov--Eliashberg algebras. In particular, the choice of augmentation $\varepsilon_-$ of the Chekanov--Eliashberg algebra of $\Lambda_-$ pulls back to an augmentation $\varepsilon_+=\varepsilon_-\circ\Phi_{\mathcal{L}}$ of the Chekanov--Eliashberg algebra of $\Lambda_+$.

The DGA morphism $\Phi_{\mathcal{L}}$ of the Chekanov--Eliashberg algebras after and before the surgery  was computed in \cite[Theorem 1.1]{Dimitroglou:Ambient} under the assumption that the handle-attachment is sufficiently small. This computations in particular shows that the mixed chords $c$ on $\Lambda_+ \cup \phi^T_{\partial_z}(\Lambda)$ are mapped to
$$\Phi_{\mathcal{L}}(c)=c+\sum_i r_i\mathbf{d}_i,\:\: r_i \in \kk,$$
where $\mathbf{d}_i$ are words of Reeb chords that each contain an \emph{odd} number of mixed chords of $\Lambda_- \cup \phi^T_{\partial z}(\Lambda)$, and in which every mixed chord moreover is of length strictly less than $\ell(c)$. It now follows by pure algebraic considerations that the map
$$ CF_*((\Lambda_+,\varepsilon_+),(\Lambda,\varepsilon)) \to CF_*((\Lambda_-,\varepsilon_-),(\Lambda,\varepsilon))$$
induced by linearising the DGA-morphism $\Phi_{\mathcal{L}}$ using the augmentations $\varepsilon$ and $\varepsilon_-$ (see \cite{Bilinearised} and \cite{Dimitroglou:Cthulhu}) is an action-preserving isomorphism of the Floer complexes as claimed.
\qed

\bibliographystyle{alphanum}
\bibliography{references}

\end{document}